\documentclass{article}
\usepackage[a4paper, total={6in, 8in}]{geometry}
\usepackage[T1]{fontenc}
\usepackage[utf8]{inputenc}
\usepackage[english]{babel}
\usepackage[version=3]{mhchem} % Package for chemical equation typesetting
\usepackage{siunitx} % Provides the \SI{}{} and \si{} command for typesetting SI units
\usepackage{graphicx} % Required for the inclusion of images
\usepackage[sort,square,numbers]{natbib}
\usepackage{amsmath} % Required for some math elements 
\usepackage{amssymb}
\usepackage{physics}
\usepackage{array}
\usepackage{amsthm}
\usepackage{subfig}
\usepackage{authblk}

\newtheorem{theorem}{Theorem}
\newtheorem{remark}{Remark}
\newtheorem{lemma}{Lemma}
\newtheorem{proposition}{Proposition}
\newtheorem{definition}{Definition}
\newtheorem{corollary}{Corollary}

\graphicspath{{./images/}}

\title{Solutions of Quadratic First-Order ODEs applied to Computer Vision Problems\\ Plane Curve Reconstruction \\ UAH} % Title

\author[1]{David Casillas-Perez}
\author[1]{Daniel Pizarro}
\author[2]{Adrien Bartoli}
\author[1]{Manuel Mazo}
\affil[1]{Department of Electronic, University of Alcal\'a, Alcal\'a de Henares, Spain}
\affil[2]{ISIT - CNRS/Universit\'e d'Auvergne, Clermont-Ferrand, France}
\affil[ ]{\textit{\{casillasperezdavid,dani.pizarro,adrien.bartoli\}@gmail.com, manuel.mazo@uah.es}}

\date{\today} % Date for the report

\begin{document}

\maketitle % Insert the title, author and date

% If you wish to include an abstract, uncomment the lines below
\begin{abstract}
This article is a study about the existence and the uniqueness of solutions of a specific quadratic first-order ODE that frequently appears in multiple reconstruction problems.
It is called the \emph{planar-perspective equation} due to the duality with the geometric problem of reconstruction of planar-perspective curves from their modulus.
Solutions of the \emph{planar-perspective equation} are related with planar curves parametrized with perspective parametrization due to this geometric interpretation.
The article proves the existence of only two local solutions to the \emph{initial value problem} with \emph{regular initial conditions} and a maximum of two analytic solutions with \emph{critical initial conditions}.
The article also gives theorems to extend the local definition domain where the existence of both solutions are guaranteed.
It introduces the \emph{maximal depth function} as a function that upper-bound all possible solutions of the \emph{planar-perspective equation} and contains all its possible \emph{critical points}.
Finally, the article describes the \emph{maximal-depth solution problem} that consists of finding the solution of the referred equation that has maximum the depth and proves its uniqueness.
It is an important problem as it does not need initial conditions to obtain the unique solution and its the frequent solution that practical algorithms of the state-of-the-art give.
\end{abstract}

%----------------------------------------------------------------------------------------
%	SECTION 1
%----------------------------------------------------------------------------------------
\section{Introduction}
\label{sec:introduction}
% Introduccion a las ODE's PDE's en vision y reconstrucction.
Several computer vision problems are described in terms of Partial Differential Equations (PDEs) and Ordinary Differential Equations (ODEs).
In reconstruction of objects, as soon as we abandon the rigid reconstruction problems, PDEs and ODEs play a fundamental role describing the shape deformations of the object to reconstruct.
To such an extend that we can reduce the deformable reconstruction problem to solving specific PDEs and ODEs.
For instance, focusing on 3D reconstruction of deformable surfaces, they are used to impose local constrains in isometric, conformal and equiareal Shape-from-Template (SfT) problems~\cite{Bartoli2015}.
In Non-Rigid Structure-from-Motion (NRSfM)~\cite{Parashar2016,Parashar2018} PDEs are used to the same purpose. Considering the 3D curve reconstruction problem, the PDEs defines possible stretchings too.
Again in the ancient reconstruction problem Shape-from-Shading (SfS), they appear to define formation models of images as Lambertian models. In case of planar curve reconstruction problem,  ODEs are used to establish 1DSfT~\cite{Gallardo2015}.

% Necesidad de estudiar Teoria de ODEs y PDEs
This huge list of example can be extended.
The study of different methods for solving differential equations is required and theorems that guarantees the uniqueness of reconstruction are needed.

% Critica a la forma de resolver los problemas de reconstruccion anteriores
Some of the cited problems are called well-posed problems in the sense of they are determined itself without adding constrains.
It is the case of isometric SfT where the problem is described by $3$ PDEs in $3$ unknown variables.
The uniqueness of solutions is derived from the fact that each point is algebraically determined by the equations~\cite{Bartoli2015}.
However, the majority of these problems need additional constrains to be solved uniquely.
Most of the reconstruction algorithms calculate one of all possible solutions without taking care of the others, normally the smoothest one, see the \emph{maximal-depth solution problem} Section~\ref{sec:maxDepthSol}.

% Vinculación de las ODEs con la cámara de perspectiva
ODEs and PDEs we handle in reconstruction problems are strongly linked with the camera projection model used to compound the images. Perspective camera model are the most frequent camera projection model because of the practical applications.
Pinhole cameras capture scenes projecting light rays based on this model.
The nature of the projection constrains the analysis of the existence and the uniqueness of the reconstructions we can make from the image captured. Orthographic projection model is other camera projection model very studied in computer vision because of its simplicity and the well behavior in the sense of uniqueness.
Against the problem of reconstruction of planar-orthographic curves from its modulus, it derives to an ODE whose solutions are all displaced versions of one from the $x$-axis and/or mirrored due to the \emph{concave-convex} nature of the equation.
The existence of only two local solution from the equivalent \emph{planar-orthographic equation} is guaranteed for both regular and critical points.
We want to answer if the \emph{planar-perspective equation} presents equivalent results.

%Que da el artículo
The present article studies the existence and uniqueness of solutions of a specific quadratic first-order ODE that repeats in many of the previous mentioned reconstruction problems.
It is called \emph{planar-perspective equation} due to the nature of the ODE is closely connected with the planar curve reconstruction problem projected by the perspective model knowing its modulus.
Nevertheless, the conclusions which derive from it have strong repercussions in 3D reconstruction of curves and surfaces, specially in the appearance of multiple solutions.
The articles gives theoretical proofs of existence of solutions and study the uniqueness of the problem. It brings the following specific contributions.
\emph{i)} we established a mathematical framework to define the called \emph{planar-perspective equation} that is a quadratic first-order ODE related with the planar curve reconstruction problem from its modulus.
Within this framework, we will study the existence and the uniqueness of solutions of this specific ODE.
\emph{ii)} we show that in absence of \emph{initial conditions}, there is a dense set of solutions that fulfill the equation and consequently, multiple planar curves could be chosen as candidates.
\emph{iii)} we prove the existence of two local solution of the ODE adding a \emph{regular initial conditions} and the existence of a maximum of two analytical solutions adding a \emph{critical initial condition}.
The difference comes from the analysis that we develop for both different scenarios. This is our major contribution of the article. We remark the dissimilarities with respect the orthographic case.
Also, we give some formulas to extend the region where uniqueness theorem is guaranteed and to find the local bounds.
\emph{iv}) we study the called \emph{maximal curve} which contains all \emph{critical points} of solutions and impose an upper-bound for all of them.
\emph{v)} Finally, we define the problem of finding the \emph{maximal depth solution} which is the solution of the \emph{planar-perspective equation} with maximum depth.
The problem is common in the literature~\cite{Bartoli2015}.
Practical reconstruction algorithms, most of them based on the minimization of a energy function, return this solution in absence of \emph{initial conditions} due to the fact that it is also the smoothest one, as we will prove, and minimization methods try to regularize solutions.
We will prove that the \emph{maximal depth solution} is unique.

\section{Notation}
\label{sec:notation}
We use italic upper-case math calligraphy to define general sets as curves $\mathcal{C}$.
Parametrization of curves are represented by the pair $(I,X_\rho)$ where the first element $I$ represents the domain of the parametrization and the second one $X_\rho$ the function whose image is the curve and which takes $I$ as its domain.
The subscripts $\rho$ represents the \emph{depth function} on which it depends since we will works with perspective parametrization.
A brief study of different perspective parametrizations used in computer vision is detailed in Appendix~\ref{sec:appPersParam}.
Throughout the article, we will use the called polar perspective parametrization for representing curves due to the calculus simplification it brings, see Appendix~\ref{sec:appPersParam}.

We will work in the planar Euclidean space. The symbol $\lVert \cdot \rVert_2$ refers to the Euclidean norm. We use different symbols to refer to the derivative operator depending on the mathematical object involved.
We use a single quote $'$ to express the derivation of a parametrization $X'(\theta)$.
If we define a ODE we use the conventional symbol $\frac{d}{d\theta}$.
In the iterative process of generating Taylor series in Section~\ref{sec:singCond}, we use the following specific notation: $\rho_{i)}^j$ = ``\emph{i}th derivative of $\rho$ raised to the \emph{j}th power''.

\section{Problem Statement}
Let us start with the geometric interpretation of the problem, see Figure~\ref{fig:original1DSfT}.
\begin{figure}[h]
	\centering
	\includegraphics[width=0.8\textwidth]{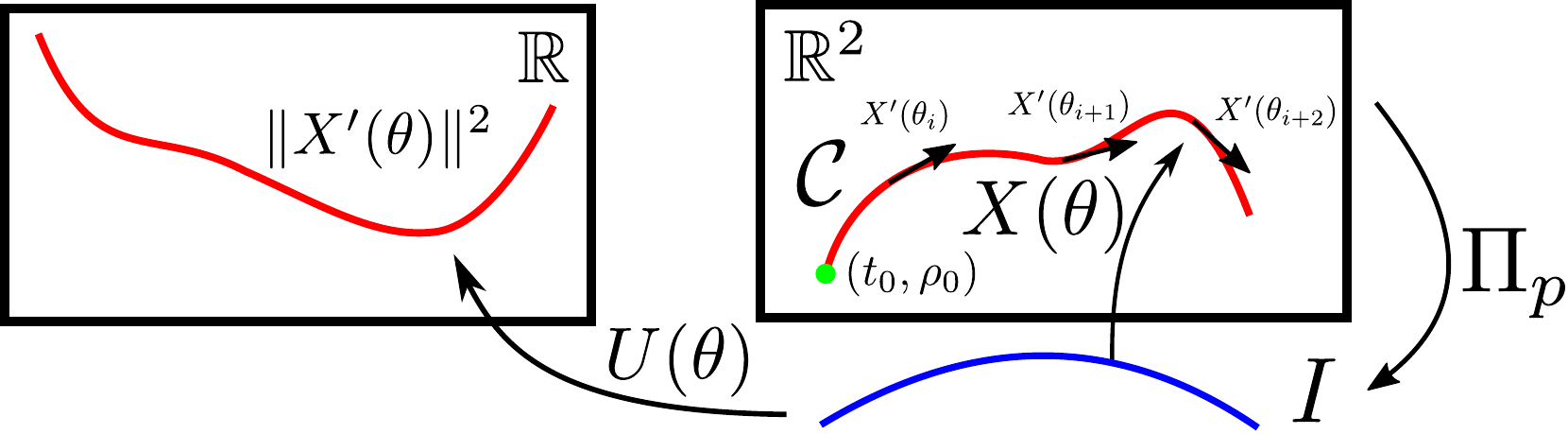}
	\caption{Geometric interpretation of the problem.}
	\label{fig:original1DSfT}
\end{figure}
Let $\mathcal{C}$ be a regular curve placed in the positive semi-plane of $\mathbb{R}^2$ parametrized by the \emph{polar perspective parametrization} $(I,X_\rho)$:
\begin{align}
X:I &\longmapsto \mathbb{R}\cross\mathbb{R}^+ \nonumber \\
\theta   &\longmapsto \rho(\theta)(\cos(\theta),\sin(\theta))
\label{eq:ppparam}
\end{align}
where the domain $I\subset[0,\pi]\subset\mathbb{R}$ is an interval and $\rho:I \longmapsto \mathbb{R}^+$ is called the \emph{depth function} associated with this parametrization, see Appendix~\ref{sec:appPersParam}.

Calculating the derivative of this parametrization we obtain the following expression of the velocity vector:
\begin{equation}
\label{eq:velocityVector}
X'(\theta) = \left(\frac{d\rho}{d\theta}\cos(\theta) - \rho\sin(\theta), \frac{d\rho}{d\theta}sin(\theta) + \rho\cos(\theta)\right)
\end{equation}
Computing the square of its modulus which we called $U$, we address to the next quadratic first-order ODE:
\begin{equation}
\lVert X'(\theta) \rVert^2_2 = \left(\frac{d\rho(\theta)}{d\theta}\right)^2 + \rho^2(\theta) = U(\theta)
\end{equation}
Assuming known this function $U(\theta)=\lVert X'(\theta) \rVert^2_2$, we wonder if we can reconstruct the curve $\mathcal{C}$ with only this information or maybe adding some constrain as a point of the curve.
Observe that the problem of reconstruction of a planar-curve from its modulus become in solving an specific quadratic first-order ODE.
We can keep in mind this duality between the geometric objects curves and the solutions of the previous ODE.

Putting the geometrical interpretation to one side, we formally describe the problem as follows.
Let $U:I \longmapsto \mathbb{R}^+$ be a function of class $\mathcal{C}^1(I,\mathbb{R}^+)$ defined in the interval $I$ except for a finite number of points $\mathcal{A}\subset I$ where the function is of class $\mathcal{C}^\infty(\mathcal{A},\mathbb{R})$ and where the first-order derivative of the $\rho$ function is null.
The function $U(\theta)$ has the geometric interpretation of being the square of the modulus of the velocity vector~\eqref{eq:velocityVector} of the curve $\mathcal{C}$ calculated through the parametrization $(I,X_\rho)$ as we mentioned before. It means that $U(\theta) = \lVert X_{\rho}'(\theta)\rVert^2_2$ and we assume that it is known.
Consider the next quadratic first-order ordinary differential equation (\emph{ODE}):
\begin{equation}
\left(\frac{d\rho(\theta)}{d\theta}\right)^2 + \rho^2(\theta) = U(\theta),
% \varepsilon\frac{d\rho}{dt}^2 + \rho^2 = U,
\label{eq:1}
\end{equation}
where $\rho(\theta)$ is the unknown function that we have to find which is positive or null $\rho(\theta)\geq0$.
The variable $\theta$ is the independent variable.
We want to answer firstly if the ODE~\eqref{eq:1} is enough to define a unique solution.
Observe that as two variables are involved $\rho$ and $\frac{d\rho}{d\theta}$ and there is only one equation that related them, there will be multiple solutions to the ODE~\eqref{eq:1}.
Figure~\eqref{fig:1DSfTrecta} illustrates that in absence of any other constrains there are multiple solutions that fulfills the ODE. 

We ask now for the Initial Value Problem (\emph{IVP})~\eqref{eq:pvi1} (also called the Cauchy problem) composed by the equation~\eqref{eq:1} (also called \emph{planar-perspective reconstruction equation} or simply \emph{planar-perspective equation}) with the \emph{Initial Condition} (IC) $\rho(\theta_0) = \rho_0$ in a neighborhood $\theta_0\in J \subset I$.
\begin{equation}
\left\lbrace
\begin{array}{l}
\left(\frac{d\rho(\theta)}{d\theta}\right)^2 + \rho^2(\theta) = U(\theta)\\
\rho(\theta_0) = \rho_0,
\end{array}
\right.
\label{eq:pvi1}
\end{equation}
Constraining to a specific IC, we want to answer now how many solutions there will fulfill the IVP~\eqref{eq:pvi1}.
Let start to solve the problem considering the next two examples. They illustrates that in absence of IC, multiple solutions appear and introduce the hypothesis that will prove. This is the existence of two analytical solutions given an IC.

Before entering into a formal proof, the next examples permits to visualize the concepts we manage along the article.
Besides, it act as prove of multiple solutions of ODE~\eqref{eq:1} appear if no initial conditions constrains the problem.

%\subsection{Geometric Interpretation}
%\label{sec:geometricInterpretation}
%Let $\mathcal{C}$ be a regular curve parametrizes with the polar perspective parametrization~\eqref{eq:ppparam}. Calculating the derivative of this parametrization we obtain:
%\begin{equation}
%X'(\theta) = \left(\frac{d\rho}{d\theta}\cos(\theta) - \rho\sin(\theta), \frac{d\rho}{d\theta}sin(\theta) + \rho\cos(\theta)\right)
%\end{equation}
%Computing the square of its modulus which we called $U$, we address to the next quadratic first-order ODE:
%\begin{equation}
%\lVert X'(\theta) \rVert^2 = \left(\frac{d\rho(\theta)}{d\theta}\right)^2 + \rho^2(\theta) = U(\theta),
%\end{equation}
%Assuming known this function $U$, we wonder if we can reconstruct the curve $\mathcal{C}$ with only this information or maybe adding some constrain as a point of the curve.

\section{Examples}
Figure~\ref{fig:1DSfTrecta} shows that in absence of an initial condition a dense set of multiple solutions may exist.
Green and cian curves are perspective-parametric curves which \emph{depth function} fulfills \eqref{eq:1} that differ in the \emph{initial condition} represented with magenta circles.
Once we fix an \emph{initial condition} the example shows empirical results that there are only two possible curves which pass through the initial condition.
\begin{figure}[h]
	\centering
	\includegraphics[width=0.8\textwidth]{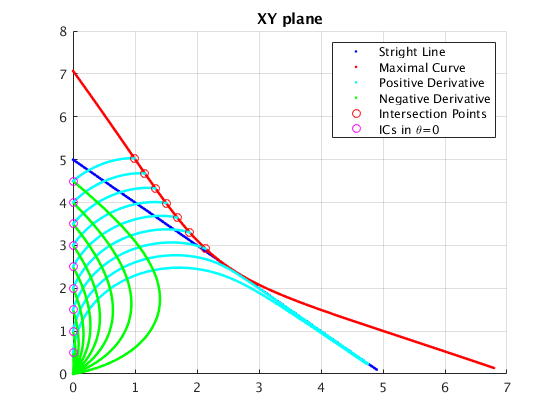}
	\caption{SfT of a straight line.
		Straight line is represented in blue color. The maximal curve is showed in red color.
		Different ICs yields to different curves one for each sign of the derivative.
		The \emph{depth function} of each curves obey the ODE~\eqref{eq:1}.}
	\label{fig:1DSfTrecta}
\end{figure}
The blue straight line represents the curve we want to reconstruct whose explicit equation is $-x + 5 = 0$. We only know the square modulus of the velocity vector $\vec{v}$.
Magenta points represent different regular ICs in $\theta=0$. There are two possible curves that passes throw each of them, one with positive derivative and one with negative derivative.
These two curves are represented with cian and green colors respectively.
The curve whose increase monotonically in depth (cian curves) may intersect or not with the the red curve called the \emph{maximal curve} formally defined in section~\ref{sec:maximalCurve}. Briefly, this curve is obtained by vanishing the term $\frac{d\rho}{dt} = 0$ in the original equation~\eqref{eq:1} and solving for $\rho$.
It has the property of containing all the points of solutions $\rho$ of the equation~\eqref{eq:1} with null first-derivative, that means the critical points of the ODE~\eqref{eq:1}.
All of the curves fulfills the ODE~\eqref{eq:1} and they are projected into the same points at least before intersecting the \emph{maximal curve}.
We need more information that only fulfilling the ODE to recover our desired blue straight line.
Adding an IC that belongs to the blue straight line is required.

Figure~\ref{fig:1DSfTrectaNullCI} showed the reconstruction of the previous straight line $-x + 5 = 0$.
Observe that given the initial point $(0,5)$ as an IC, the decreasing curve (green) match with the blue curve we want to reconstruct.
We have needed at least one point of the blue straight line to reconstruct it locally.
Observe that as long as we descend through the curve, it tangentially intersects  the red \emph{maximal curve} in a point $p$ with null first-derivative (respect to its depth function)(we will call it a critical point).
This kind of point has the normal of the curve parallel to the optical ray (the line that joins $p$ with the origin).
Critical points have the properties of create branches of solutions as we can see in the figure breaking the uniqueness of the problem.
At this point the curve splits in two, one piece continue decreasing its depth (green), an the other start to increase it (cian).
The cian and green curves are the two solutions that we obtain with different regular ICs that share a piece of blue straight line.
We also observe that there exist only $2$ possible analytic solutions or $4$ piecewise functions $C^1$ function if we mix all the branches in the critical point.

\begin{figure}[h]
	\centering
	\includegraphics[width=0.8\textwidth]{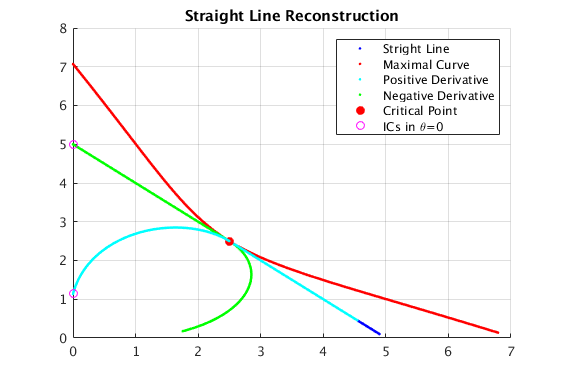}
	\caption{SfT of a straight line in a null initial condition.}
	\label{fig:1DSfTrectaNullCI}
\end{figure}

\begin{remark}
	It is possible to have curves without any non-null initial condition. For instance, the red straight line in figure~\ref{fig:1DSfTrectaNullCI} with a positive grow has non-null initial conditions.
\end{remark}
\section{Main Result}
\label{sec:mainResult}

We want to prove that the \emph{IVP}~\eqref{eq:pvi1} composed by the equation~\eqref{eq:1} and any IC $(\theta_0,\rho_0)$, inside the domain of the equation, has a maximum of two possible solutions $\rho_1$ and $\rho_2$ if the IC is \emph{regular} and a maximum of two analytical solutions if the IC is critical.
In other words, speaking in geometric terms, given a regular point in the plane $p=(\theta_0,\rho_0)$ compatible with the equation~\eqref{eq:1}, there are a maximum of two curves that fulfill the equation~\eqref{eq:1} and pass through the point $p=(\theta_0,\rho_0)$ and given a singular point, a maximum of two analytical curves. In the case critical ICs, we look for analytic solutions.

We divide the proof in two parts.
The first one assumes that the IC is \emph{regular} that means that the IC has non-null first-order derivative. The second one assumes that IC is \emph{critical} or with a null first-order derivative.

\subsection{Case 1: Regular ICs}
\label{sec:regularICs}

Solving the equation~\eqref{eq:1} for $\frac{d\rho}{d\theta}$, we obtains the next two explicit ODEs that differ from the sign.
\begin{equation}
\frac{d\rho}{d\theta} = \pm\sqrt{U(\theta)-\rho^2},
\label{eq:2}
\end{equation}
Therefore, we have the next two explicit IVP with the same initial condition, one with the positive sign of the square root and the negative one.
\begin{table}[htb!]
	\centering
\begin{tabular}{|lc|lc|}
	\hline
	1) &
	$
		\left\lbrace
		\begin{array}{l}
		\frac{d\rho}{d\theta} = +\sqrt{U(\theta)-\rho^2}\\
		\rho(\theta_0) = \rho_0,
		\end{array}
		\right.
	$
	& 2) &
	$
		\left\lbrace
		\begin{array}{l}
		\frac{d\rho}{d\theta} = -\sqrt{U(\theta)-\rho^2}\\
		\rho(\theta_0) = \rho_0,,
		\end{array}
		\right.
	$\\
	\hline
\end{tabular}
\caption{The two explicit Initial Value Problems that match with the ODE~\eqref{eq:2}. The first one has positive sign of the derivative. The second one has the negative}
\label{tab:pvi1}
\end{table}
By the hypothesis, we assume that $U(\theta_0) > \rho_0^2$ that implies that the derivatives of $\rho$ function is not vanished in IC, see equation~\eqref{eq:2}. Therefore, we obtain two different values of the derivative of $\rho$ in $\theta_0$ that we call $\frac{d\rho(\theta_0)}{d\theta}=\alpha$ and $\frac{d\rho(t_0)}{dt}=-\alpha$ with $\alpha>0$. The value of the derivative in the IC is not null and both values have equal modulus and different sign.

Using the \emph{Picard-Lindelöf} theorem for each IVP of the Table~\ref{tab:pvi1} with $\rho(\theta_0)=\rho_0$ as IC guarantees the existence and the uniqueness of both solutions $\rho_1(\theta)$ and $\rho_2(\theta)$ respectively at least in a neighbourhood around the IC.
The solution of the first IVP $\rho_1(\theta)$ grows in this neighbourhood and the second one $\rho_2(\theta)$ decreases.
Proposition~\ref{th:SufConPL} proves that the \emph{Picard-Lindelöf} conditions to call the theorem are satisfied.
\begin{proposition}
	Considering the right side of the explicit ODE~\eqref{eq:2} as a function of the variables $\theta$ and $\rho$:
	\begin{equation}
		f(\theta,\rho) = \pm\sqrt{U(\theta)-\rho^2}
	\end{equation}
	\label{th:SufConPL}
	The function $f(\theta,\rho)$ is Lipschitz continuous in the second variable in an interval $J$ around the IC.
\end{proposition}
\begin{proof}
	The partial derivative $\frac{\partial f}{\partial \rho}$ exist around an interval $J$ of the IC due to the hypothesis $U(\theta_0) > \rho_0^2$ and it has the next expression.
	\begin{equation}
		\frac{\partial f}{\partial \rho} = \frac{\mp\rho}{\sqrt{U(\theta)-\rho^2}}
	\end{equation}
	The partial derivative $\frac{\partial f}{\partial \rho}$ is continuous in the interval $J$. Calling \emph{Weierstrass theorem} in any closed interval $I\subset J$ proves that $\frac{\partial f}{\partial \rho}$ is bounded and consequently is \emph{Lipschitz continuous} in $J$.
\end{proof}
The interval $J$ where the \emph{Lipschitz continuous} property is satisfied can be spread out until a critical point appeared. Consequently local solutions are unique in the interval $J$ which finish at critical points.
Solutions of the two explicit IVP shows in Table~\ref{tab:pvi1} are monotonic as its derivatives verify $\frac{d\rho}{dt} \ge 0$ and $\frac{d\rho}{dt} \le 0$, respectively. The first one gives solutions monotonically increasing.
The second one monotonically decreasing. This means that until solutions reach critical points, solutions always grows or always decrease, see Figure~\ref{fig:monotonicSol}.
\begin{figure}[h]
	\centering
	\includegraphics[width=0.7\textwidth]{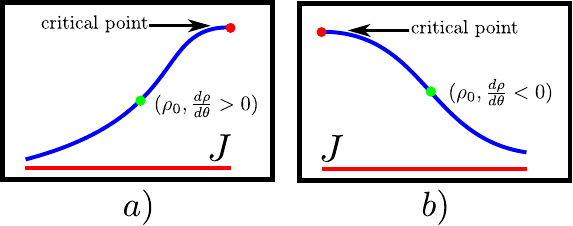}
	\caption{Monotonic behaviour of solutions of explicit IVP of the Table~\ref{tab:pvi1} until reach a critical point.}
	\label{fig:monotonicSol}
\end{figure}

Proposition~\ref{th:SufConPL} involves satisfactory consequences:
\begin{remark}
	If ODE~\eqref{eq:1} does not present critical points in its domain $I$ then there are only two possible global solutions of the IVP~\eqref{eq:pvi1}, one monotonic increasing and one monotonic decreasing. The existence and uniqueness local domain $J$ of the Proposition~\ref{th:SufConPL} can be expanded to take up the global domain $I$.
\end{remark}
\begin{remark}
	Focusing on an interval of the global domain  $J\subset I$ where Proposition~\ref{th:SufConPL} is satisfied, a consequence of the proposition is that local solutions of the IVP\eqref{eq:pvi1} in $J$ never intersects between them. Only ICs that belong to the same local solution could be do it.
\end{remark}

\begin{remark}
	If  critical points exist in the definition domain $I$, multiple solutions could be appear connecting local solutions from two consecutive local domains $J_1$ and $J_2$ that share the critical point and match in the the first-order derivative $\rho_1(\theta_c) = \rho_2(\theta_c)$ and $\frac{d\rho_1(\theta_c)}{d\theta} = \frac{d\rho_2(\theta_c)}{d\theta} = 0$.
\end{remark}

Proposition~\ref{th:SufConPL} proves the importance that critical points have in the study of the existence and the uniqueness of solutions.
Finding all possible critical points is required in order to recognize all possible solutions, or where branches arise.
The \emph{maximal curve} solve the problem of discovering al critical points as we will see in Section~\ref{sec:maximalCurve}.

Now, we wonder what happen when the initial condition is a critical point, that means, a point where the first-order derivative vanishes.
Given a critical point, there exist a maximum of two analytic curves that pass through this point? Notice that if the initial condition $\rho(\theta_0)=\rho_0$ is a critical point, the equation~\eqref{eq:1} becomes algebraic.
Unfortunately, we cannot use the \emph{Picard-Lindelöf} theorem to prove that there exist a maximum  two analytic solutions as critical points does not satisfied the \emph{Lipschitz continuous} property, see Proposition~\ref{th:SufConPL}.
Instead of it, we follow the next steps.

\subsection{Case 2: Critical ICs}
\label{sec:singCond}

Firstly, we differentiate both sides of equation~\eqref{eq:1}.
We use the notation $\rho_{i)}^j$ to mean the \emph{i}th derivative of $\rho$ raised to the \emph{j}th power.
\begin{equation}
	2\left(\rho_{1)}\rho_{2)} + \rho_{0)}\rho_{1)}\right) = U_{1)}
\label{eq:3}
\end{equation}
This equation is second-order.
Besides, we observe that the highest-order of the equation is not quadratic.
Therefore, any bifurcation appears.
Since we assume that the first-order derivative vanishes at the initial condition, equation~\eqref{eq:3} is an identity and necessary it must fulfill $U_{1)}=0$.
On the contrary, the equation would not have a solution.
We cannot solve the IVP with this equation and the initial condition $(t_0,\rho_0)$ with null derivative.
We will exploit this property when we define the \emph{maximal curve} in Section~\ref{sec:maximalCurve}, a useful curve to find all possible critical ICs.
We repeat the process and differentiate both sides of the previous equation again:
\begin{equation}
2\left(\rho_{2)}^2+\rho_{1)}\rho_{3)} + \rho_{1)}^2+\rho_{0)}\rho_{2)}\right)= U_{2)}
\label{eq:4}
\end{equation}
This equation is third-order.
The coefficient of the highest-order term is the same in both equations and equals to $2\rho_{1)}$.
As we assume that the first-order derivative in the initial condition is null, equation~\eqref{eq:4} becomes a quadratic in $t=t_0$:
\begin{equation}
2\left(\rho_{2)}^2 + \rho_{0)}\rho_{2)}\right)= U_{2)}
\label{eq:5}
\end{equation}
Equation~\eqref{eq:5} produces a bifurcation or a ramification. Solving the equation in $\rho_{2)}$, we obtain two different values for the second-order derivatives in the critical IC due to its quadratic character. 
\begin{equation}
\beta_{l} = \frac{-\rho_{0)}\pm\sqrt{\rho_{0)}^2+2 U_{2)}}}{2} \quad l\in\{1,2\}
\end{equation}
We will refer to this pair of values of the second-order derivative as $\beta_1$ and $\beta_2$.
In this point, we assume that the discriminant $\Delta=\rho_{0)}^2 + 2U_{2)}\geq0$ in $t=t_0$ to obtain real roots. This condition forces that $U_{2)}$ lies on the interval next interval:
\begin{equation}
U_{2)}\in\left[\frac{-\rho_{0)}^2}{2},\infty\right)
\end{equation}
The value of the discriminant $\Delta$ determines the sign of the second-order derivative $\rho_{2)}$.	
\begin{equation}
\left\lbrace
\begin{array}{l}
(\beta_1,\beta_2)\in (-\infty,-\rho_0]\cross[0,\infty) \qquad U_{2)} \ge 0\\
(\beta_1,\beta_2)\in (-\rho_0,-\frac{2\rho_{0)}}{3}]\cross[-\frac{\rho_{0)}}{3},0) \qquad U_{2)} \in [\frac{-4\rho_{0)}^2}{9},0)\\
(\beta_1,\beta_2)\in (-\frac{2\rho_{0)}}{3},\frac{-\rho_{0)}^2}{2}]\cross[-\frac{\rho_{0)}}{2},-\frac{\rho_{0)}}{3}) \qquad U_{2)} \in [\frac{-\rho_{0)}^2}{2},\frac{-4\rho_{0)}^2}{9}),
\end{array}
\right.
\label{eq:signDerivative}
\end{equation}
 Again, we repeat the process and differentiate both sides of the equation:
\begin{equation}
2\left(3\rho_{2)}\rho_{3)} + \rho_{1)}\rho_{4)} + 3\rho_{1)}\rho_{2)} + \rho_{0)}\rho_{3)}\right) = U_{3)}
\label{eq:6}
\end{equation}
This equation is fourth-order. At this point, we prove the existence of a maximum of only two possible analytic solutions.
We can calculate the value of the third-order derivative for each value of the critical IC (there are two because of the bifurcation) $(\rho_0,0,\beta_1)$ and $(\rho_0,0,\beta_2)$. Vanishing the first-order derivative, we obtain the next third-order equation:
\begin{equation}
2\left(3\rho_{2)}\rho_{3)} + \rho_{0)}\rho_{3)}\right) = U_{3)}
\label{eq:7}
\end{equation}
For each IC $(\rho_0,0,\beta_1)$ y $(\rho_0,0,\beta_2)$ we obtain a unique value of the third-order derivative (due to the fact that the coefficient of the highest-order derivative in the consecutive iterations is linear, in this case $\rho_{3)}$, at least in the non-degenerate cases).
We will prove at the end of the article.
If we follow this procedure, we can build a Taylor series for each initial condition and calculates $a=(\rho_0,0,\beta_1,...)$ and $b=(\rho_0,0,\beta_2,...)$ in $\theta_0$.
Assuming that the function $U(\theta)$ is analytic at the critical points, its Taylor series converges in a neighborhood of $\theta_0\in K$. As we are looking for the analytical solutions of the ODE in a neighborhood around the IC and the operations that involves the equation (squares, sums, derivatives) are friendly with the analytic property of the functions, then, the Taylor series of the function $U(\theta)$ is compatible with the existence of analytic solutions of the ODE~\eqref{eq:1}.
That means that given an analytic solution of the ODE~\eqref{eq:1} necessary $U(\theta)$ will be analytic.
Solutions $\rho_1(\theta)$ and $\rho_2(\theta)$ that generate from the method are analytic assuming convergence.

\begin{equation}
\begin{split}
\rho_1(\theta) = \sum_{i=0}^{\infty} \frac{a_i}{i!}(\theta-\theta_0)^i\quad t\in K\\
\rho_2(\theta) = \sum_{i=0}^{\infty} \frac{b_i}{i!}(\theta-\theta_0)^i\quad t\in K
\end{split}
\end{equation}

We now prove the existence of a maximum of only two solutions extending the interval $K$ as long as we reach another critical point.
As the obtained function is analytic in $\theta_0$, the Taylor series converges in a neighborhood $K$ around $\theta_0$.
For any point $\theta\in K,\theta\neq \theta_0$, the first-order derivative is not null and, consequently, we are in the conditions of the regular IC, see Section~\ref{sec:regularICs}.
We can apply the \emph{Picard-Lindellöf} theorem to guarantee the existence of only two possible analytic solutions that passes through the critical point that belong to the new analytical.

There exist a particular degenerate case where all the coefficients of the Taylor series are null except $\rho_0$.
In this case, the function $\rho(\theta)=\rho_0$ is a constant. It is linked with the circle whose radius is constant.

Th next points are important to generalize the proof to all order of the derivatives.
In the previous proof, we assumed that the bifurcation appeared in the second-order derivative, but, in general, the bifurcation may appear in the \emph{n}th-order derivative.
\begin{enumerate}
	\item Whatever the number of times we differentiate the equation, the coefficient of the highest-order of the derivative is always $2\rho_{1)}$. We prove it using the \emph{general Leibniz rule} that has the next expression with our notation:
	\begin{equation}
		(f\cdot g)_{n)} = \sum_{k=0}^n {n \choose k} f_{k)}g_{n-k)}
	\end{equation}
	Substituting $f=g=\rho_{1)}$, we find that:
	\begin{equation}
		x_{n)}=(\rho_{1)}\cdot \rho_{1)})_{n)} = \sum_{k=0}^{n}  {n \choose k} \rho_{k+1)}\rho_{n-k+1)}
		\label{eq:x}
	\end{equation}
	We do the same for the non-derivative term:
	\begin{equation}
		y_{n)}=(\rho \cdot \rho)_{n)} = \sum_{k=0}^{n}  {n \choose k} \rho_{k)}\rho_{n-k)}
	\end{equation}
	Now, equation~\eqref{eq:1} can be rewritten as follow:
	\begin{equation}
		x + y = U
		\label{eq:8}
	\end{equation}
	The successive derivatives can be written as:
	\begin{equation}
		x_{n)} + y_{n)} = U_{n)}\qquad n\in\mathbb{N}
		\label{eq:9}
	\end{equation}
	Fixing the iteration $n$, we always find that the highest-order derivative ($(n+1)$-order) is obtained from the terms $k=0$ and $k=n$ of equation~\eqref{eq:x}, leading to:
	\begin{equation}
	\left({n \choose 0}\rho_{1)}\rho_{n+1)} + {n \choose n}\rho_{n+1)}\rho_{1)}\right) = 2\rho_{_1)}\rho_{n+1)}
	\end{equation}
	
	Consequently, this term vanishes at the critical points, reducing the order of the differential equation by one.
	
	\item We now show some expressions for $x_{n)}$, $y_{n)}$ from equation~\eqref{eq:9}.

	$x_{n)}$:
	\begin{equation}
	    \begin{array}{l|l}
        n = 0 & \rho_{1)}^2\\
		n = 1 & 2{1\choose 0}\rho_{1)}\rho_{2)}\\
		n = 2 & 2{2\choose 0}\rho_{1)}\rho_{3)} + {2\choose 1}\rho_{2)}^2\\
		n = 3 & 2{3\choose 0}\rho_{1)}\rho_{4)} + 2{3\choose 1}\rho_{2)}\rho_{3)}\\
		n = 4 & 2{4\choose 0}\rho_{1)}\rho_{5)} + 2{4\choose 1}\rho_{2)}\rho_{4)} + {4\choose 2}\rho_{3)}^2\\
		n = 5 & 2{5\choose 0}\rho_{1)}\rho_{6)} + 2{5\choose 1}\rho_{2)}\rho_{5)} + 2{5\choose 2}\rho_{3)}\rho_{4)}\\
		n = 6 & 2{6\choose 0}\rho_{1)}\rho_{7)} + 2{6\choose 1}\rho_{2)}\rho_{6)} + 2{6\choose 2}\rho_{3)}\rho_{5)} + {6\choose 3}\rho_{4)}^2\\
		\vdots & \vdots\\
		n = i-1 & 2{i-1\choose 0}\rho_{1)}\rho_{i)} + 2{i-1\choose 1}\rho_{2)}\rho_{i-1)} +\cdots + 2{i-1\choose\frac{i-4}{2}}\rho_{\frac{i-2}{2})}\rho_{\frac{i+4}{2})} + 2{i-1\choose\frac{i-2}{2}}\rho_{\frac{i}{2})}\rho_{\frac{i+2}{2})}\\
		n = i & 2{i\choose 0}\rho_{1)}\rho_{i+1)} + 2{i\choose 1}\rho_{2)}\rho_{i)} + \cdots + 2{i\choose\frac{i-2}{2}}\rho_{\frac{i}{2})}\rho_{\frac{i+4}{2})} + {i\choose\frac{i}{2}}\rho_{\frac{i+2}{2})}^2\\
		n = i+1 & 2{i+1\choose 0}\rho_{1)}\rho_{i+2)} + 2{i+1\choose 1}\rho_{2)}\rho_{i+1)} + \cdots + 2{i+1\choose\frac{i-2}{2}}\rho_{\frac{i}{2})}\rho_{\frac{i+6}{2})} + 2{i+1\choose\frac{i}{2}}\rho_{\frac{i+2}{2})}\rho_{\frac{i+4}{2})}\\
		\vdots & \vdots
		\end{array}
	\end{equation}
	$y_{n)}$:
	\begin{equation}
	\begin{array}{l|l}
		n = 0 & \rho_{0)}^2\\
		n = 1 & 2{1\choose 0}\rho_{0)}\rho_{1)}\\
		n = 2 & 2{2\choose 0}\rho_{0)}\rho_{2)} + {2\choose 1}\rho_{1)}^2\\
		n = 3 & 2{3\choose 0}\rho_{0)}\rho_{3)} + 2{3\choose 1}\rho_{1)}\rho_{2)}\\
		n = 4 & 2{4\choose 0}\rho_{0)}\rho_{4)} + 2{4\choose 1}\rho_{1)}\rho_{3)} + {4\choose 2}\rho_{2)}^2\\
		n = 5 & 2{5\choose 0}\rho_{0)}\rho_{5)} + 2{5\choose 1}\rho_{1)}\rho_{4)} + 2{5\choose 2}\rho_{2)}\rho_{3)}\\
		n = 6 & 2{6\choose 0}\rho_{0)}\rho_{6)} + 2{6\choose 1}\rho_{1)}\rho_{5)} + 2{6\choose 2}\rho_{2)}\rho_{4)} + {6\choose 3}\rho_{3)}^2\\
		\vdots & \vdots\\
		n = i-1 & 2{i-1\choose 0}\rho_{0)}\rho_{i-1)} + 2{i-1\choose 1}\rho_{1)}\rho_{i-2)} +\cdots + 2{i-1\choose\frac{i-4}{2}}\rho_{\frac{i-4}{2})}\rho_{\frac{i+2}{2})} + 2{i-1\choose\frac{i-2}{2}}\rho_{\frac{i-2}{2})}\rho_{\frac{i}{2})}\\
		n = i & 2{i\choose 0}\rho_{0)}\rho_{i)} + 2{i\choose 1}\rho_{1)}\rho_{i-1)} + \cdots + 2{i\choose\frac{i}{2}}\rho_{\frac{i-2}{2})}\rho_{\frac{i+2}{2})} + {i\choose\frac{i}{2}}\rho_{\frac{i}{2})}^2\\
		n = i+1 & 2{i+1\choose 0}\rho_{0)}\rho_{i+1)} + 2{i+1\choose 1}\rho_{1)}\rho_{i)} + \cdots + 2{i+1\choose\frac{i-2}{2}}\rho_{\frac{i-2}{2})}\rho_{\frac{i+4}{2})} + 2{i+1\choose\frac{i}{2}}\rho_{\frac{i}{2})}\rho_{\frac{i+2}{2})}\\
		\vdots & \vdots
	\end{array}
	\end{equation}
	We now show some iterations:
	\begin{equation}
	\begin{array}{l|l}
	n = 1 & x_{1)} + y_{1)} = 2{1\choose 0}\rho_{1)}\rho_{2)} + 2{1\choose 0}\rho_{0)}\rho_{1)} + = U_{1)}\\
	n = 2 & x_{2)} + y_{2)} = 2{2\choose 0}\rho_{1)}\rho_{3)} + {2\choose 1}\rho_{2)}^2 + 2{2\choose 0}\rho_{0)}\rho_{2)} + {2\choose 1}\rho_{1)}^2= U_{2)}\\
	n = 3 & x_{3)} + y_{3)} = 2{3\choose 0}\rho_{1)}\rho_{4)} + 2{3\choose 1}\rho_{2)}\rho_{3)} + 2{3\choose 0}\rho_{0)}\rho_{3)} + 2{3\choose 1}\rho_{1)}\rho_{2)} = U_{3)}\\
	n = 4 & x_{4)} + y_{4)} = 2{4\choose 0}\rho_{1)}\rho_{5)} + 2{4\choose 1}\rho_{2)}\rho_{4)} + {4\choose 2}\rho_{3)}^2 + 2{4\choose 0}\rho_{0)}\rho_{4)} + 2{4\choose 1}\rho_{1)}\rho_{3)} + {4\choose 2}\rho_{2)}^2 = U_{4)}\\
	n = 5 & x_{5)} + y_{5)} = 2{5\choose 0}\rho_{1)}\rho_{6)} + 2{5\choose 1}\rho_{2)}\rho_{5)} + 2{5\choose 2}\rho_{3)}\rho_{4)} + 2{5\choose 0}\rho_{0)}\rho_{5)} + 2{5\choose 1}\rho_{1)}\rho_{4)} + 2{5\choose 2}\rho_{2)}\rho_{3)} = U_{5)}\\
	n = 6 & x_{6)} + y_{6)} = 2{6\choose 0}\rho_{1)}\rho_{7)} +\cdots+ 2{6\choose 2}\rho_{3)}\rho_{5)} + {6\choose 3}\rho_{4)}^2 + 2{6\choose 0}\rho_{0)}\rho_{6)} +\cdots+ 2{6\choose 2}\rho_{2)}\rho_{4)} + {6\choose 3}\rho_{3)}^2= U_{6)}\\
	\vdots & \vdots\\
	n = i-1 & x_{i-1)} + y_{i-1)} = 2{i-1\choose 0}\rho_{1)}\rho_{i)} +\cdots+ 2{i-1\choose\frac{i-2}{2}}\rho_{\frac{i}{2})}\rho_{\frac{i+2}{2})} + 2{i-1\choose 0}\rho_{1)}\rho_{i)} +\cdots+ 2{i-1\choose\frac{i-2}{2}}\rho_{\frac{i}{2})}\rho_{\frac{i+2}{2})}= U_{i-1)}\\
    n = i & x_{i)} + y_{i)} = 2{i\choose 0}\rho_{1)}\rho_{i+1)} +\cdots+ {i\choose\frac{i}{2}}\rho_{\frac{i+2}{2})}^2 + 2{i\choose 0}\rho_{0)}\rho_{i)} +\cdots+ {i\choose\frac{i}{2}}\rho_{\frac{i}{2})}^2= U_{i)}\\
	n = i+1 & x_{i+1)} + y_{i+1)} = 2{i+1\choose 0}\rho_{1)}\rho_{i+2)} +\cdots+ 2{i+1\choose\frac{i}{2}}\rho_{\frac{i+2}{2})}\rho_{\frac{i+4}{2})} + 2{i+1\choose 0}\rho_{0)}\rho_{i+1)} +\cdots+ 2{i+1\choose\frac{i}{2}}\rho_{\frac{i}{2})}\rho_{\frac{i+2}{2})}= U_{i+1)}\\
	\vdots & \vdots
	\end{array}
	\end{equation}
	
	\item The first time that we differentiate equation~\eqref{eq:1} and evaluate the critical IC $(\theta_0,\rho_0)$, we obtain an identity and cannot compute the second-order derivative.
	If we differentiate again, we obtain the second-order equation~\eqref{eq:5} in $\rho_{2)}$.
	We can obtain two different real roots or one double root.
	Differentiating again and substituting the value $\rho_{1)}=0$ yields equation~\eqref{eq:7}.
	It is possible to solve for $\rho_{3)}$ if we know all the previous derivatives.
	The only constraint is that the coefficient of this derivative has to be non-null $2(\rho_{0)} + 3\rho_{2)}) \neq 0$ in order to resolve it.
	But, there is a value of $\rho^{(2)}$ that vanishes the coefficient of this derivative and obeys $2(\rho_{0)} + 3\rho_{2)})=0$.
	We can prove that after the second iteration, the highest-order of the resultant equation $\rho_{i)}$ after substituting $\rho_{1)}=0$ has the next coefficient $\alpha_i(i)$:
	\begin{equation}
		\alpha_{i}(i) = 2(\rho_{0)}+(i+1)\rho_{2)}) \qquad i\geq2
	\end{equation}
	If this coefficient does not vanish for any iteration, then, we achieve all the coefficients of the Taylor series.
	We can also express the condition of non-vanishing as:
	\begin{equation}
		\rho_{2)} \neq \frac{-\rho_{0)}}{(i+1)} \qquad \forall i\geq2, i\in\mathbb{N}
	\end{equation}
	It is possible to give sufficient conditions that guarantees that the coefficient of the highest-order derivative does not vanish for any iteration. Lemma~\ref{lm:suffcientCondition} and~\ref{lm:suffcientCondition2} gives two of these possible sufficient conditions.
	\begin{lemma}
		\label{lm:suffcientCondition}
		If the value of $\rho_{2)}$ lies on the next subset:
		\begin{equation}
		\rho_{2)} \in (-\infty, \frac{-\rho^{0)}}{3}) \cup [0,\infty),
		\end{equation}
		then the coefficient of the highest-order derivative will never vanish for any iteration.
	\end{lemma}
	Lemma~\ref{lm:suffcientCondition} is satisfied when the value of the second-order derivative of $U$ belongs to the next set:
	\begin{equation}
	U_{2)} \in \left[\frac{-\rho_{0)}^2}{2},\frac{-4\rho_{0)}^2}{9}\right) \cup \left[0,\infty\right)
	\end{equation}
		\begin{lemma}
		\label{lm:suffcientCondition2}
		If the ratio $\frac{\rho_{2)}}{\rho_{0)}} \not \in \mathbb{Q}$, that means it is irrational, then the coefficient of the highest-order derivative will never vanish for any iteration $\alpha_{i}(i)\neq0\forall i\ge2$.
	\end{lemma}
	We can prove also that if for some iteration the coefficient of the highest-order vanishes, thus, this coefficient will never go back to vanish in successive iterations. Consequently, we obtain a system with more variables than equations. In this case, each derivatives depends on parameter forming a 1-parameter family of coefficients.
	
	We call degenerated cases those one whose coefficient of the highest-order $\alpha_m(i)$ vanishes for some iteration $i$. The set of degenerates cases $\mathcal{B}$ is a countable infinite and is contained in the next interval:
	\begin{equation}
		\rho_{2)} \in \mathcal{B}=\left\{\frac{-\rho^{0)}}{(i+1)}\;\big|\; i\in\mathbb{N},\:i\geq 2\right\}\subset I_{\mathcal{B}}=\left[\frac{-\rho^{0)}}{3},0\right)
	\end{equation}
		
	\item Assuming that we can obtain all the coefficients, we can build a Taylor serie and impossing analiticity in critical points.
	Taylor series of functions in analytical points converges in a neighbourhood of the point to the function (see the book~\cite{}).
	As we assumed that critical points are isolated points we can use \emph{Picard-Lindelöf} theorem to any point of the neighbourhood of the point proves the existence and the uniqueness as we show in case of derivative non-null.
	\emph{Quod erat demonstrandum}
\end{enumerate}

\section{The Maximal Curve}
\label{sec:maximalCurve}

We showed in Section~\eqref{eq:1} the importance of critical points in the study of the uniqueness of multiple solutions.
The maximal depth function and the maximal curve associated play an important role in order to solve the IVP~\eqref{eq:pvi1}.
Knowing its behaviour is transcendent to find all possible global solutions to the IVP~\eqref{eq:pvi1}.

\begin{definition}
    Let define the maximal depth function $\rho_{max}$ as the function that result of vanishing the term $\frac{d\rho}{d\theta}$ of the equation~\eqref{eq:1}. Therefore,
    \begin{align}
	    \rho_{max}:I & \longmapsto \mathbb{R}^+ \nonumber\\
	    \theta   &\longmapsto \rho_{max}(\theta)= +\sqrt{U(\theta)}
    \end{align}
    
\end{definition}
\begin{definition}
	The maximal curve $\mathcal{C}_{max}$ is the image of the polar perspective parametrization $(I,X_{\rho_{max}})$ whose depth function associated is $\rho_{max}$.
\end{definition}
The maximal depth function $\rho_{max}$ is no solution in general of the ODE~\eqref{eq:1}.
Effectively, regions $J\subset I$ where the maximal function $\rho_{max}$ is strictly monotonically increasing $\frac{dU}{dt} > 0$ or strictly monotonically decreasing $\frac{dU}{dt} < 0$ cannot fullfil the ODE:
\begin{equation}
\left(\frac{d\rho_{max}}{d\theta}\right)^2 + \rho_{max}^2 = \left(\frac{d\rho_{max}}{d\theta}\right)^2 + U > U.
\end{equation}
Only those points of $\rho_{max}$ whose first-derivatives vanish fulfills the ODE~\eqref{eq:1}.
These are extrema or inflection points of the maximal depth function as the next theorem proves.
\begin{theorem}
	Extrema or inflection points of the maximal curve are critical points of the ODE~\eqref{eq:1}.
\end{theorem}
\begin{proof}
	Extrema and inflection points have null first-order derivative.
	Consequently, they fulfills the equation.
	\begin{equation}
	\left(\frac{d\rho_{max}}{d\theta}\right)^2 + \rho_{max}^2 = \rho_{max}^2 = U.
	\end{equation}
\end{proof}
The importance of this theorem comes from the fact that critical points of the ODE are contained in this maximal curve.
We can compute where the critical points of the equation are located before calculate its solutions.

The next corollary guarantees two global solutions to the IVP~\eqref{eq:pvi1} looking for the first-derivative of the maximal depth function.
\begin{corollary}
	If the maximal depth function $\rho_{max}$ has no extrema or inflection points, then, there will be local solutions only two solution of the IVP\eqref{eq:pvi1} in all the definition domain $I$.
	There will be a unique monotonic increased solution and a unique monotonic decreased solution.
\end{corollary}
\begin{proof}
	The ODE has no critical points in the definition domain $I$ of the ODE.
	Consequently, we can extend the inteval $J$ where  \emph{Picard-Lindelöf} is satisfied until covering the definition domain $I$.
\end{proof}
Now we study the behavior of a critical point that is a minimum of the maximal depth function $\rho_{max}$.
We prove that there exist only a unique solution that has a minimum in the minimum of the maximal depth function. Figure~\eqref{fig:minimumSol} helps us to prove the next theorem.
\begin{theorem}
	If the IC of the IVP~\eqref{eq:pvi1} is a critical point related with a minimum of the maximal depth function $\rho_{max}$, then, there exits a unique local solution $\rho$ of the IVP~\eqref{eq:pvi1} with a minimum in the IC.
	\label{th:minSol}
\end{theorem}
\begin{proof}
	Assuming that there exist two monotonic increasing solutions $\rho_1(\theta)$, $\rho_2(\theta)$ defined at the right side of the critical point $(\theta_0,\rho_0)$ noted with $J^+=[\theta_0,\epsilon]$ that obey $\rho_1(\theta_0)=\rho_2(\theta_0)=\rho_0$.
	Assuming that $\rho_1>\rho_2$ in $(\theta_0,\epsilon]$, as both as solutions of the ODE, they fulfills $0<\frac{d\rho_1}{d\theta} < \frac{d\rho_2}{d\theta}$ in $(\theta_0,\epsilon]$.
	Consequently, $\rho_2$ grows faster than $\rho_2$ and as they start growing at the same point $(\theta_0,\rho_0)$, $\rho_2>\rho_1$ that is a contradiction.
	We can repeat the same argument to the left side $J^-$ of the IC.
	As a conclusion, there are only one solution of the ODE $\rho_1=\rho_2=\rho$ with a minimum at the same place.
\end{proof}
\begin{figure}[h]
	\centering
	\includegraphics[width=0.6\textwidth]{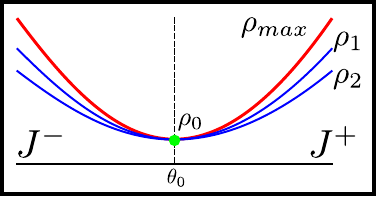}
	\caption{Proof the existence to a unique monotonically increasing function at the right side. There are only a unique monotonically decreasing function in the left side of the minimum}.
	\label{fig:minimumSol}
\end{figure}
Maxima of the maximal depth function $\rho_{max}$ do not present this behavior.
They may exist a \emph{converge cone} with an infinite dense set of solutions with a maximum at the same maximum of the depth function, see Section~\ref{sec:convCone} and Figure~\ref{fig:ex1}-\ref{fig:ex2}. Theorem~\ref{th:minSol} can be reformulated to prove the unique solution of the branch that match with a minimum if the critical IC is an inflection point.

The maximal depth function $\rho_{max}$ has another important property.
It forms an upper-bound of the sets of all possibles solutions of the ODE~\eqref{eq:2} as the next Theorem proves.
Because of this reason, we call this function as maximal depth function.
\begin{theorem}
	The maximal depth function obeys:
	\begin{equation}
	\rho_{max}(\theta) \ge \rho(\theta),
	\end{equation}
	where $\rho(\theta)$ is any solution of the ICV.
	The equality succeed in  critical points of the ODE.
\end{theorem}
\begin{proof}
	As all solutions $\rho$ of the equation~\eqref{eq:2} fulfill the equation itself and this one is a sum of squares, we obtain the next inequality.
	 
	\begin{equation}
		\rho^2 \le \left(\frac{d\rho}{d\theta}\right)^2 + \rho^2 = U =  \rho_{max}^2.
	\end{equation}
	The equality is arisen only in critical points where $\frac{d\rho}{d\theta} = 0$.
\end{proof}
This is an important point in order to search the solution of the ODE with the maximum depth.
The solution will share these critical points and will be tangent to the maximal depth function.

\section{The Maximal Solution}
\label{sec:maxDepthSol}
A problem that arises from the study of this particular ODE~\eqref{eq:1} consists in finding the solution with more depth. It means that is the farthest from the coordinate origin for all value $\theta\in I$.
The problem is equivalent to look for the solution with least curvature or the most smooth.
Figure~\ref{fig:maximalSol} illustrates the problem of finding the \emph{maximal solution} of the ODE.
\begin{figure}[h]
	\centering
	\includegraphics[width=0.6\textwidth]{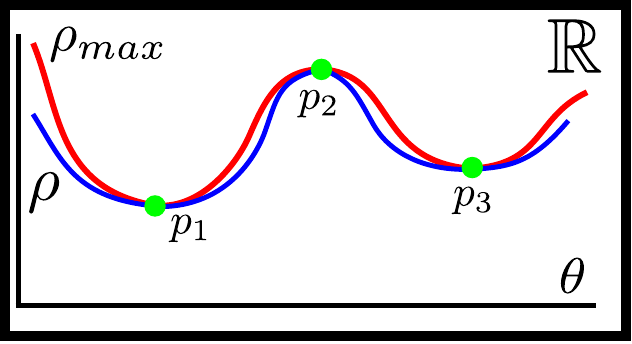}
	\caption{The red curve represents a maximal depth function with three extrema points. The maximal solution passes through these three points and maximize the depth for each point or equivalently minimize the curvature.}
	\label{fig:maximalSol}
\end{figure}
Let start defining the set af all possible solutions to the ODE~\ref{eq:1}.
\begin{definition}
	Let $\mathcal{S}$ be the set of all possible solutions of the ODE~\eqref{eq:1}.
	\begin{equation}
		\mathcal{S} = \left\{\rho \mid \left(\frac{d\rho}{d\theta}\right)^2 + \rho^2 = U \right\}
	\end{equation}
	As the set of solution $\mathcal{S}$ is closed, upper-bounded by the maximal depth function $\rho_{max}(\theta)$ as we indicated in Section~\ref{sec:maximalCurve} and lower-bounded by $\rho(\theta)=0$, we define the maximal depth solution $\rho_M$ as the solution which obey:
	\begin{equation}
		\rho_M(\theta) \ge \rho(\theta) \quad \forall \rho \in \mathcal{S}, \forall \theta
	\end{equation}
\end{definition}
The maximal depth solution $\rho_M$ passes through all critical points because this points are depth maxima.
\begin{equation}
\rho_M(\theta) = \rho_{max}(\theta), \quad \theta \in \mathcal{A}
\end{equation}
For addressing the whole problem, we consider the partition $P=\mathcal{A}$ of the global domain $I$ built with from all possible critical points. Consequently, $I=[\theta_0,\theta_1]\cup\cdots\cup[\theta_{|\mathcal{A}|},\theta_{|\mathcal{A}|+1}]$.
We first start reducing the problem to the following Boundary Value Problem (BVP) considering one of the intervals $J_i=[\theta_{i}, \theta_{i+1}]$ generated through the partition $P$. 
\begin{equation}
\left\lbrace
\begin{array}{l}
\left(\frac{d\rho(\theta)}{d\theta}\right)^2 + \rho^2(\theta) = U(\theta)\\
\rho(\theta_i) = \rho_{i}, \; \rho(\theta_{i+1})=\rho_{i+1} \quad \theta\in J_i \; , \; i \in \{\,j\;|\;0<j<|\mathcal{A}|\}
\end{array}
\right.
\label{eq:bvp1}
\end{equation}
Consequently, we are in the situation that Figure~\ref{fig:increaseDecreaseMaxDepth} shows.
We ask for the solutions of the ODE that passes through both critical points.
\begin{figure}[h]
	\centering
	\includegraphics[width=0.7\textwidth]{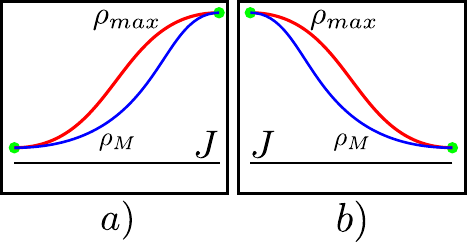}
	\caption{Reduced BVP.}
	\label{fig:increaseDecreaseMaxDepth}
\end{figure}
Notice that there exist only one solution to this problem calling Theorem~\ref{th:minSol}.
One of the critical point is a minimum of the maximal depth function, the first one in case of Figure~\ref{fig:increaseDecreaseMaxDepth}.a) and the second one in case of Figure~\ref{fig:increaseDecreaseMaxDepth}.b).

Taking all critical points of the partition $P$ in consecutive pairs we obtain a unique solution for each individual BVP. Concatenation the solutions we generate a solution that cover all the definition domain $I$ except to the extreme intervals $J_0$ and $J_{|\mathcal{A}|}$.
Each solution of an individual BVP fits with the next one due to the fact because they share a critical point that fit in first-order derivative.
Consequently, the solution composed joining solutions of this BVP is of the class $C^1$.

\section{Convergence cones}
\label{sec:convCone}

The previous section proves the existence of a maximum of two analytic solution of the IVP around a critical point.
But multiple non-analytic solutions could appear. We show that when the second-order derivatives extracted by our method are both negatives or null, then, there exists a convergent region delimited by the two analytic solutions where the IVP is satisfied.

Consider the next IVP~\eqref{eq:ex1}:
\begin{equation}
	\left\lbrace
	\begin{array}{l}
	\frac{d\rho}{d\theta} = f(\theta,\xi)\\
	\rho(0) = 1,
	\end{array}
	\right.
	\label{eq:ex1}
\end{equation}
where
\begin{equation}
f:\left[0,\frac{\pi}{2}\right]\cross\left[0,1\right] \longmapsto \mathbb{R}
\end{equation}
is defined as $f(\theta,\xi) = -\sqrt{1-\xi^2}$. Observe that it is one of the IVP derived from the original one~\ref{eq:pvi1} substituting $U(\theta)=1$, see Table~\ref{tab:pvi1}.
In this example, the IC $\rho(0)=1$ is a critical point for the ODE equation.
Moreover, as the ODE is autonomous $f(\theta,\xi) = f(\xi)$ all ICs $\rho(\theta) = 1$ are critical points.
The IVP has two analytic solutions.
The constant function $x_1(\theta) = 1$ is an analytical solution of the problem and the maximal curve too.
We can find the other analytical solution $x_2(\theta) = cos(\theta)$ solving the separable ODE.
Our method yields to the same solutions if it is applied to the quadratic equation~\eqref{eq:ex1_q}.
\begin{equation}
\left(\frac{d\rho(\theta)}{d\theta}\right)^2 + \rho^2(\theta) = 1,
% \varepsilon\frac{d\rho}{dt}^2 + \rho^2 = U,
\label{eq:ex1_q}
\end{equation}
In this case, we obtain the Taylor series of $x_1(\theta)$ and $x_2(\theta)$:
\begin{equation}
\left\lbrace
\begin{array}{l}
x_1(\theta) = 1\\
x_2(\theta) = \sum^{\infty}_{n=0} \frac{(-1)^n x^{2n}}{(2n)!},
\end{array}
\right.
\label{eq:ex1_series}
\end{equation}
Observe that as the ODE is autonomous, solutions of the shifted IVP, are also solutions of original IVP~\eqref{eq:ex1}.
We can built new piecewise solutions from the analytic ones to sweep all the space and.
Therefore, the next non-analytic piecewise functions are solutions of the IVP too.
\begin{equation}
x_{\theta_0}(\theta) = 
\begin{cases}
1 & \theta < \theta_0\\
\cos(\theta-\theta_0) & \theta \ge \theta_0,
\end{cases}
\label{eq:ex1_nAnalSol}
\end{equation}
where $\theta_0 \in \left(0, \frac{\pi}{2}\right)$.
Figure~\ref{fig:ex1} shows both analytical (the blue and red curves) and non-analytical solutions (green curves) of the IVP~\eqref{eq:ex1}. Only $\mathcal{A} = \{(\theta,1) \mid \theta \in [0,\frac{\pi}{2}] \}$ are critical points for de ODE.
\begin{figure}[h]
	\centering
	\includegraphics[width=0.7\textwidth]{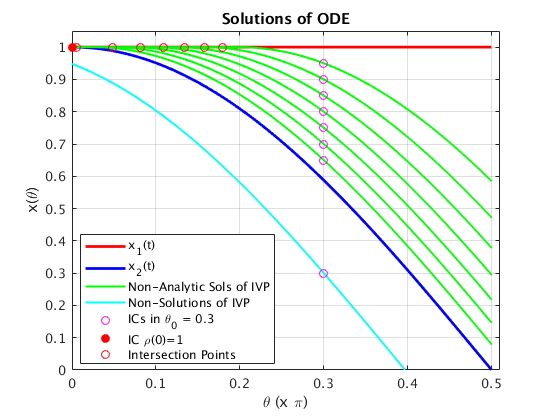}
	\caption{Solutions of the IVP. Analytic solutions of the IVP $x_1(\theta)$ and $x_2(\theta)$ are represented in red and blue respectively.
		Non-analytic solutions are drawn in green. Observe that non-analytic solutions are made shifting $x_2(\theta)$ to other IC (represented in magenta) and merging with $x_1(t)$ when they reach it.}
	\label{fig:ex1}
\end{figure}

The region located between $x_1(\theta)$ and $x_2(\theta)$ is called convergence cone.
\begin{equation}
\mathcal{D} = \left\{ (\theta,\xi) \mid \theta \in \left(0,\frac{\pi}{2}\right], x_2(\theta) < \xi < x_1(\theta) \right\}
\end{equation}
Theorem~\ref{th:sandwitch} proves that if we choose any IC $(\theta_0,\rho_0) \in \mathcal{D}$ then there exists a solution $x_s(\theta)$ of the IVP~\eqref{eq:ex1_2}

\begin{equation}
	\left\lbrace
	\begin{array}{l}
	\frac{d\rho}{d\theta} = f(\theta,\xi)\\
	\rho(\theta_0) = \rho_0,
	\end{array}
	\right.
	\label{eq:ex1_2}
\end{equation}
that obey $x_1(\theta) < x_s(\theta) < x_2(\theta)$ where $\theta\in\left(0,\frac{\pi}{2}\right]$ and only match in the critical points $(0,1)$.

\begin{theorem}
	Consider the next IVP 
	\begin{equation}
		\left\lbrace
		\begin{array}{l}
		\frac{d\rho}{d\theta} = f(\theta,\rho)\\
		\rho(\theta_0) = \rho_0,
		\end{array}
		\right.
		\label{eq:th}
	\end{equation}
	where the function $f$
	\begin{equation}
		f:\left[\theta_0,b\right]\cross\mathbb{R}^+ \longmapsto \mathbb{R}
	\end{equation}
	is continuous in $[\theta_0,b]$, \emph{Lipschitz} in $(\theta_0,b]$ and satisfied $f(\theta_0,\rho_0)=0$. 
	Assuming that there are two different solutions $x_1(\theta)$ and $x_2(\theta)$ of the IVP~\eqref{eq:th}, choosing a new IC  $(\theta_*,\rho_*)$ from the set $\mathcal{D}$
	\begin{equation}
		\mathcal{D} = \left\{ (\theta,\xi) \mid \theta \in \left(\theta_0,b\right], x_1(\theta) < \xi < x_2(\theta) \right\}
	\end{equation}
	then the unique solution $x_*(\theta)$ of the new IVP~\eqref{eq:th2} that is built taken the previous one~\eqref{eq:th} and changing the IC $(\theta_*,\rho_*)$
	\begin{equation}
		\left\lbrace
		\begin{array}{l}
		\frac{d\rho}{d\theta} = f(\theta,\rho)\\
		\rho(\theta_*) = \rho_*,
		\end{array}
		\right.
		\label{eq:th2}
	\end{equation}
	%is a solutions of the IVP~\eqref{eq:th} and fullfil $x_1(\theta) < x_*(\theta) < x_2(\theta)$ for $\theta\in\left(\theta_0,b\right] $ and matches with $x_1$, $x_2$ in the critical point $x_1(\theta_0) = x_*(\theta_0) = x_2(\theta_0)$
	is also a solution of the ICV~\eqref{eq:th}. Besides, it fullfils $x_1(\theta) < x_*(\theta) < x_2(\theta)$ for $\theta\in\left(\theta_0,b\right]$ and it intersect them in the IC $x_1(\theta_0) = x_*(\theta_0) = x_2(\theta_0)$. The set of points $\mathcal{D}$ is called convergence cone.
	\label{th:sandwitch}
\end{theorem}
\begin{proof}
	The \emph{Cauchy-Peano} theorem guarantees the existence of at least one solutions.
	As the function $f$ is \emph{Lipschitz continuous} in the domain except to in the IC, the IVP~\eqref{eq:th2} has a unique solutions $x_*(\theta)$ in a neighborhood of the IC $(\theta_*,\rho_*)$ that cannot intersect $x_1(\theta)$, $x_2(\theta)$ because if not the uniqueness would break.
	As solutions $x_1(\theta)$, $x_2(\theta)$ intersect into the only critical point and $x_*(\theta)$ is strictly bounded by these two solutions, calling \emph{Squeeze theorem}, $x_*(\theta)$  have to intersect both of them in the IC.
\end{proof}
Particularizing the Theorem~\ref{th:sandwitch} to our case, we have $f(\theta,\rho) = \pm \sqrt{U-\rho}$ that is continuous in all the domain and it is \emph{Lipthchitz continuous} except for the discrete set of points where $f$ vanishes.
The maximums of the $U(\theta)$ function have the properties of gives two negative values for the second-derivatives of the $\rho$ Taylor series.
It means that there will be two analytic solutions $\rho_1$, $\rho_2$ associated with the negative ODE IVP~\ref{tab:pvi1}.
\begin{corollary}
	Considering the IVP~\eqref{eq:pvi1} taking a maximum of $U$ as IC, there will exist a convergence cone and, consequently, multiple non-analytic solutions.
 \end{corollary}
\begin{proof}
	Sections~\ref{sec:mainResult} proves a maximum of two analytical solutions $x_1$, $x_2$ with negative second-order derivatives in the IC. 
	It means that, both analytic solutions are related with the negative explicit IVP.
	\begin{equation}
		\left\lbrace
		\begin{array}{l}
		\frac{d\rho}{d\theta} = -\sqrt{U(\theta)-\rho^2}\\
		\rho(\theta_0) = \rho_0,,
		\end{array}
		\right.
	\end{equation}
\end{proof}
	Calling Theorem~\ref{th:sandwitch} we proof the existence of this convergence cone and the multiple solutions.
	
Now, consider another IVP~\eqref{eq:ex3_parab} similar to~\eqref{eq:ex1}:
\begin{equation}
\left\lbrace
\begin{array}{l}
\frac{d\rho}{d\theta} = g(\theta,\xi)\\
\rho(0) = \frac{\pi}{4},
\end{array}
\right.
\label{eq:ex3_parab}
\end{equation}
where $g(\theta,\xi) = -\sqrt{\frac{\pi^2}{16} -\frac{\pi^2}{128}\theta^2 - \xi^2}$. Again, the IC $\rho(0) = \frac{\pi}{4}$ is a critical point, but only this point of the domain is critical.
In the previous example there were a  dense set of critical points, the solution $x_1(\theta)=1$, due to the fact that the ODE was autonomous.
Contrary, this example presents a non-autonomous ODE.
As only one critical point appear in the equation, solutions of different IVP with the same ODE associated cannot cross it and they only can merge in the unique critical point.

Figure~\ref{fig:ex2} shows the critical curve in red color and the two analytical solutions obtained using our method around in blue. We show also that solutions of an IVP that have its IC in the convergence cone necessary converge to the critical point as we prove in Theorem~\ref{th:sandwitch}.

\begin{figure}[h]
	\centering
	\subfloat[]{\includegraphics[width=0.5\textwidth]{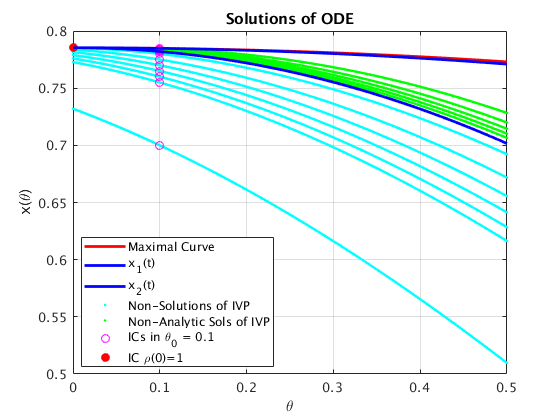}}
	\subfloat[]{\includegraphics[width=0.5\textwidth]{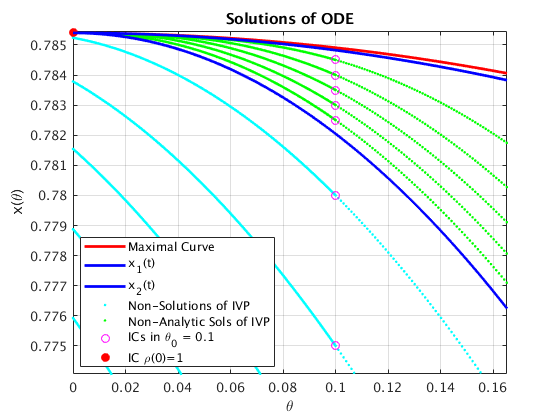}}
	\caption{. The analytic solutions of the IVP~\eqref{eq:ex3_parab} $x_1(\theta)$ and $x_2(\theta)$ are represented with blue lines.
		Both create a convergence cone where non-analytic solutions (represented with green) are also solutions to the IVP and converge to IC $(0,\frac{\pi}{4})$.
		Cian curve are solutions of the ODE but not of the ICV because they does not reach the IC $(0,\frac{\pi}{4})$. The maximal curve are represented in red color.}
	\label{fig:ex2}
\end{figure}

\appendix
\section{Perspective Parametrizations of Curves}
\label{sec:appPersParam}

There are three main manner in which the perspective parametrizations of regular curves $\mathcal{C}\subset\mathbb{R}^2$ appears in computer vision problems. All of them are equivalent but working with a specific one brings advantages of simplifying calculus respect to the others.
For our case, the problem is simplified drastically when use the called \emph{polar perspective parametrization}.
This parametrization maps each angle $\theta$ between the interval $[0,\pi]$ into each point of the curve.
We can imagine that the projection set is the circumference with radius the unit (that it is parametrized by an angle), see Figure~\ref{fig:persParam}.

\begin{figure}[h]
	\centering
	\includegraphics[width=0.9\textwidth]{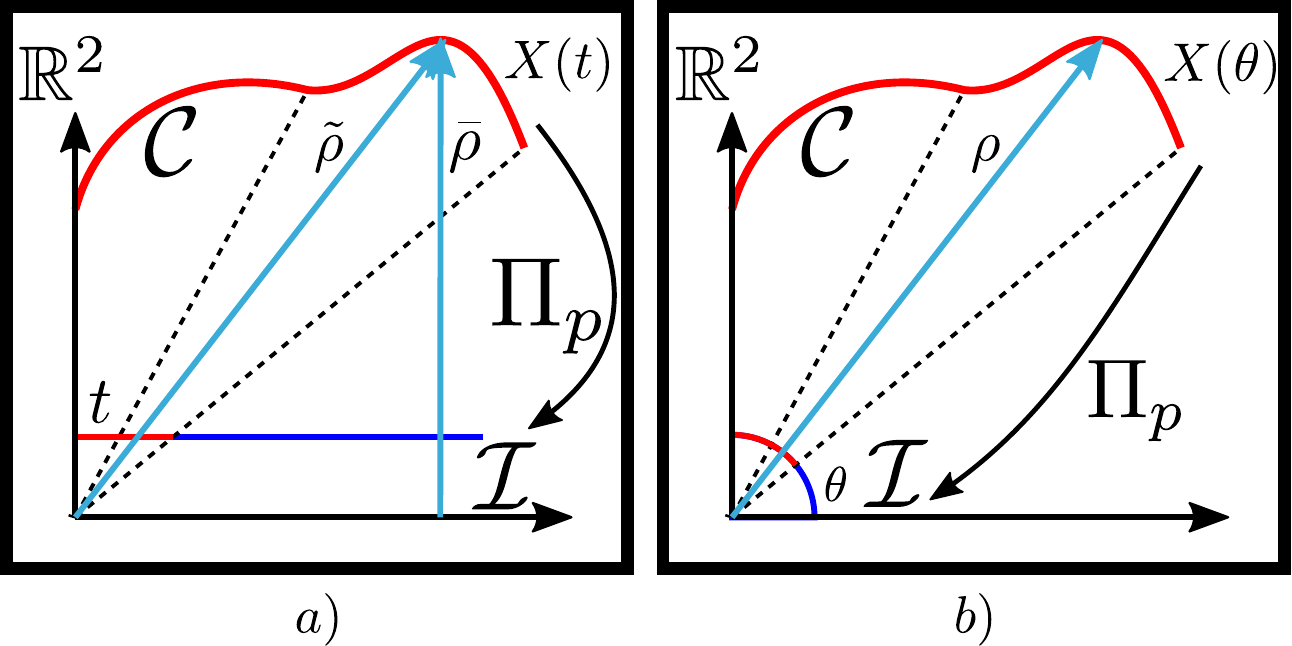}
	\caption{Most common perspective parametrizations in computer vision. Figure a) shows two common \emph{Cartesian perspective parametrizations}. Figure b) shows the \emph{polar perspective parametrization}.}
	\label{fig:persParam}
\end{figure}

The most common perspective parametrization $(I,X_{\bar{\rho}})$ of a curve $\mathcal{C}$ is expressed in terms of a \emph{depth function} $\bar{\rho}:I\longmapsto\mathbb{R}^+$ that measures the distance of each point of the curve $\mathcal{C}$ with respect to the $x$-axis.
The domain of the depth function $t\in I\subset\mathbb{R}$ is an interval of the called \emph{image line}.
\begin{align}
X_{\bar{\rho}}:I &\longmapsto \mathbb{R}\cross\mathbb{R}^+\nonumber\\
t   &\longmapsto \bar{\rho}(t)(t,1)
\end{align}

The vector $(t,1)$ indicate the direction of projection for each point of the image line.
The modulus of this vector $\lVert (1,t) \rVert_2 = \sqrt{1+t^2}$ that is not of length $1$ for all direction.
If we normalize this vector and multiply and divide this value to the previous parametrization we obtain a new parametrization.
\begin{align}
X_{\tilde{\rho}}:I &\longmapsto \mathbb{R}\cross\mathbb{R}^+\nonumber\\
t   &\longmapsto \frac{\tilde{\rho}(t)}{\sqrt{1+t^2}}(t,1),
\end{align}
where $\tilde{\rho} = \bar{\rho}\sqrt{1+t^2}$. Now, the \emph{depth function} $\tilde{\rho}$ measures the distance between each point of the curve $\mathcal{C}$ and the coordinates origin. The previous parametrizations are called \emph{Cartesian perspective parametrizations}.
Both parametrization are represented in Figure~\ref{fig:persParam}.a).
Observe that the parametrization $(I,X_{\tilde{\rho}})$ is radial.

Other perspective parametrizations of curves works with angles as a parameter.
The next parametrization is called \emph{polar perspective parametrization} and use $\theta$ as a variable.
\begin{align}
X_{\rho}:\left[0,\pi\right] &\longmapsto \mathbb{R}\cross\mathbb{R}^+\nonumber\\
\theta   &\longmapsto \rho(\theta)\left(\cos(\theta),\sin(\theta)\right),
\end{align}
The \emph{depth function} $\rho:\left[0,\pi\right]\longmapsto\mathbb{R}^+$ measures again the distance between each point of the curve $\mathcal{C}$ and the coordinate origin but its domain changes to the interval $\left[0,\pi\right]$.
It is easy to see that we can convert $X_{\tilde{\rho}}$ into $X_{\tilde{\rho}}$ using the next change of variable:
\begin{align}
\eta: I &\longmapsto \left[0,\pi\right] \nonumber\\
t   &\longmapsto \theta=\arctan{\frac{1}{t}},
\end{align}
and $\rho(\theta) = \tilde{\rho}(\eta^{-1}(t))$. Figure~\ref{fig:persParam}.b) shows the \emph{polar perspective parametrization} that is also radial.

For our own purpose we will use the \emph{polar perspective parametrization}.
As we will see, working with this parametrization yields to separable ODE when we compute the norm of the tangent vector field associated with it.
This means that the term of the derivatives of the \emph{depth function} are separated by a sum with respect the \emph{depth function} that it is very convenient.

Observe that any perspective parametrization $(I,X_\rho)$ of a curve $\mathcal{C}$ is associated with a \emph{depth function} $\rho$.
We say that the curve $\mathcal{C}$ is monotonically increasing or monotonically decreasing if the \emph{depth function} $\rho$ associated is monotonically increasing or monotonically decreasing respectively.

%----------------------------------------------------------------------------------------
%	BIBLIOGRAPHY
%----------------------------------------------------------------------------------------

\bibliographystyle{abbrvnat}
\bibliography{egbib.bib}

%----------------------------------------------------------------------------------------

\end{document}